\documentclass[12pt,twoside]{article}

\usepackage{amssymb,amsthm,amsmath, amsfonts}

\begin{document}

\centerline{}

\centerline{}

\centerline {\Large{\bf The geometry of polynomial diagrams}}

\centerline{}
\centerline{\bf {Maksim Alennikov\footnote{Moscow Pedagogical State University Moscow.Russia. e-mail: am.ds.07.88@gmail.com  }}}

\centerline{}

\newtheorem{Theorem}{\quad Theorem}[section]

\newtheorem{Definition}[Theorem]{\quad Definition}

\newtheorem{Corollary}[Theorem]{\quad Corollary}

\newtheorem{Lemma}[Theorem]{\quad Lemma}

\newtheorem{Example}[Theorem]{\quad Example}

\begin{abstract}
In this paper we  introduce the concept of polynomial diagrams and its area for special polynomials.We study the properties of polynomial area diagrams. The formula for the area of an arbitrary polynomial diagrams
\end{abstract}
\vspace{0.76cm}

\section{Introduction}

In this paper, we consider polynomial diagrams.
A formula is derived which allows to find the area of an arbitrary polynomial diagrammy.Tak also studied the limit relations for sequences of squares polynomial diagrams depending on the parameter k. The same relations for the sequences studied area using finite differences of the second order. The object of study despite its simplicity has an interesting geometric structure.

Consider the polynomial of a special form (1), that we will use throughout the paper, where $q \in \mathbb{Z^{+}}$ and $q \not = 0$, k=degP(x), $n \in N$

\begin{equation}%Если нужно поставить номер у формулы
\label{primer}%метка. С ее помощью можно ссылаться на формулу
P(x)=\sum_{i=0}^{k}q^{n+i}x^{k-i} \quad
\end{equation}
$\\$
Let the map $\mathfrak{F}$ take each monomial $q^{n+i}x^{k-i}$ to integer
 point $(q^{n+i};k-i)$ namely
$\mathfrak{F}: q^{n+i}x^{k-i}\longrightarrow (q^{n+i};k-i) $, where $i=0,1,\dots,n$.
This construction is reminiscent of the Newton diagram, see for example [1]. Now we introduce the following concept.

\section{Preliminary Notes}

\begin{Definition} The flat polygon passing through the vertices at integer points $A(q^{n};0)$, $A_{i}(q^{n+i};k-i) $,where $i=1....k$ is called \textsl{Polynomial diagrams}  and is denoted by $\mathfrak{pd}$.

\end{Definition}
Now we shall give the following definition
\begin{Definition}
\textsl{Area} $\mathfrak{pd}$  is called area of  planar polygon that specifies the polynomial diagram and is denoted by
$S^{q}(\mathfrak{pd}) $
\end{Definition}
This polinomials diagrams is well defined from (1)
Take a $k=2$ and $n \in \mathbb{N}$. It follows from (1) that $P(x)=q^{n}x^{2}+q^{n+1}x+q^{n+2}$
While we concider $P(x)$ in this form if you do not agreed to and reverse.

\section{Main Results}

\begin{Theorem}
The area $S^{q}(\mathfrak{pd})$ is calculated using the following formula
\begin{equation}%Если нужно поставить номер у формулы
\label{primer}%метка. С ее помощью можно ссылаться на формулу
 S^{q}(\mathfrak{pd}) = \frac{q^{n}(q+3)(q-1)}{2}
\end{equation}

\end{Theorem}

\begin{proof}
 Polynomial  diagram for (2) a flat the polygon passes through the vertices $A(q^{n};0), A_{0}(q^{n};2),A_{1}(q^{n+1};1),A_{2}(q^{n+2};0)$ and consider the point $ \hat A_{1}(q^{n+1};0)$. We divide a $\mathfrak{pd}$ on two polygons namely: $AA_{0}A_{1} \hat A_{1}$ - rectangular trapezoid and  $A_{1}A_{2} \hat A_{2}$ -  right triangle. Then  $S^{q}(\mathfrak{pd})$ has the following form
$$S^{q}(\mathfrak{pd})=S(AA_{0}A_{1} \hat A_{1})+S(A_{1}A_{2} \hat A_{2})$$
We claim that $S(AA_{0}A_{1} \hat A_{1})= \frac{3(q^{n+1}-q^{n})}{2}$ and $S(A_{1}A_{2} \hat A_{2})=\frac{q^{n+2}-q^{n+1}}{2}$

The result is

\begin{equation}%Если нужно поставить номер у формулы
\label{primer}%метка. С ее помощью можно ссылаться на формулу
S^{q}(\mathfrak{pd})= \frac{3(q^{n+1}-q^{n})}{2} + \frac{q^{n+2}-q^{n+1}}{2}
\end{equation}
 Now, after spending the expression (3) elementary algebraic manipulations we obtain (2). This completes the proof of Theorem 1.

\end{proof}

\begin{Lemma}
If $k=2$ and $n=0$ than
\begin{equation}%Если нужно поставить номер у формулы
\label{primer}%метка. С ее помощью можно ссылаться на формулу
 S^{q}(\mathfrak{pd}) = \frac{(q+3)(q-1)}{2}
\end{equation}

\end{Lemma}
Since Lemma 3.2, it follows that (2) .

It is interesting to note that if we consider (4), for different values q, the area ratio approaches unity. Let is make these observations in the table 1.
Now we shall give the following theorem
\begin{Theorem}
 For any  $ q \in \mathbb{Z^{+}}$ and $q  \not = 0 $

 $$\lim_{q\to \infty} \frac{S^{q+1}(\mathfrak{pd})}{S^{q}(\mathfrak{pd})}=1, $$
\end{Theorem}

\begin{proof}
Taking into account  \textsl{ Corollary 1}, we obtain

$\lim\limits_{q\to \infty} \frac{S^{q+1}(\mathfrak{pd})}{S^{q}(\mathfrak{pd})}=\lim\limits_{q\to \infty}\frac{(q+1)(q+4)}{(q+3)(q-1)}=\lim\limits_{q \to \infty} \frac{q^2+5q+4}{q^2+2q-3}=1 $.
The theorem is proved.
\end{proof}

By definition, put formula
\begin{equation}%Если нужно поставить номер у формулы
\label{primer}%метка. С ее помощью можно ссылаться на формулу
\def\MYdef{\mathrel{\stackrel{\rm def}=}}
\Delta^{2}S^{q}(\mathfrak{pd})\MYdef S^{q+2}(\mathfrak{pd})-2S^{q+1}(\mathfrak{pd})+S^{q}(\mathfrak{pd})
\end{equation}

Again there is a desire  to consider the values of pairwise differences. As a result, we obtain that it is constant for all values of one.

Now we shall give the following theorem
\begin{Theorem}
For any  $ q \in \mathbb{Z^{+}}$ and $q  \not = 0 ,  \Delta^{2}S^{q}(\mathfrak{pd })=1$
\end{Theorem}
\begin{proof}
 Taking into account  \textsl{ Lemma 3.2}, we obtain  $S^{q+1}(\mathfrak{pd })=\frac{q(q+4)}{2}=\frac{q^2+4q}{2}$  and $S^{q+2}(\mathfrak{pd })=\frac{(q+5)(q+1)}{2}=\frac{q^2+6q+5}{2}$.

 Now substituting these formulas and (4) to (5).

$ \Delta^{2}S^{q}(\mathfrak{pd })= \frac{q^2+6q+5}{2} - \frac{2q^2+8q}{2} +  \frac{q^{2}+2q-3}{2}=\frac{2}{2}=1$.

 This completes the proof of theorem 3.
\end{proof}

Note that the basis of the results set out in the first part of this work were obtained purely an experimental way. The motivation of this study was the work  [Ar2005].It is also interesting to note that the result of Theorem 3, can not be generalized to the general case, it can be seen from numerical calculations.
After the above considerations we may formulate the main result of this work.

Consider the polynomial in the form (1).

\begin{Theorem}
 If $\mathfrak{pd}$ -polynomial diagram for (1), than

$$S^{q}(\mathfrak{pd})= \sum_{m=0}^{k-2} \frac{q^{n+m}(q-1)(2k-2m-1)}{2} + \frac{q^{n+k}-q^{n+k-1}}{2} $$
\end{Theorem}
\begin{proof}
The proof is completely analogous to that of Theorem 1. The basic idea is to divide a polynomial chart on $k-1$ rectangular trapezoids and one right-angled triangle. And then only remains to summarize the data space. In fact, we used the method of trapezoids.

Now give the proof in more detail.

We fix a number $m$ satisfies the condition $0 \leqslant m \leqslant k-2$ .Divide a polynomial diagram on $k-1$ rectangular trapezoid.  Consider the trapezium with vertices at points  $A(q^{n+m};k-m); B(q^{n+m};0); C(q^{n+m+1};k-m+1); D(q^{n+m+1};0)$.

Then its area is calculated by the following formula:

\begin{equation}%Если нужно поставить номер у формулы
\label{primer}%метка. С ее помощью можно ссылаться на формулу
S(ABCD)=\frac{(q^{n+m+1}-q^{n+m})(2k-2m-1)}{2}
\end{equation}

Now it is necessary to take into account that such trapezoids is $k-1$, but it should be noted that the summation index is not taking up to $k-1$ to $k-2$ as well as we need to extreme tip of the trapezoid $k-1$ coincides with the vertex of a right triangle. Then we have
$m+1=k-1$, then $m=k-2$.

It remains only to sum equation (6) and add to the resulting area of a right triangle, and hold small algebraic operations.

$$S^{q}(\mathfrak{pd})= \sum_{m=0}^{k-2} \frac{q^{n+m}(q-1)(2k-2m-1)}{2} + \frac{q^{n+k}-q^{n+k-1}}{2} $$

\end{proof}

\begin{table}[ht]
\centering
\caption{}
\label{tab:1}       % Give a unique label
%
% For LaTeX tables use
%
\begin{tabular}{lll}
\hline\noalign{\smallskip}
q & $S^{q}(\mathfrak{pd})$ &$\frac{S^{q+1}(\mathfrak{pd})}{S^{q}(\mathfrak{pd})}$  \\
\noalign{\smallskip}\hline\noalign{\smallskip}
2 &2.5 &2.4  \\
3 &6  &1.75  \\

4 &10.5  &1.52  \\

5& 16 &1.4  \\

6 & 22.5 &1.3  \\

... &  ...&  ...\\

16 & 142.5 &1.12  \\

\noalign{\smallskip}\hline
\end{tabular}
\end{table}

%%%%%%%%%%%%%%%%%%%%%%%%%%%%%%%%%%%%%%

\end{document}